\documentclass[12pt]{article}

\usepackage{amssymb}
\usepackage{amsfonts}
\usepackage{amsmath}
\usepackage{amsthm}
\usepackage{graphicx}
\usepackage[boxruled]{algorithm2e}

\def\a{{\alpha}}

\newtheorem{thm}{Theorem} 
\newtheorem{lem}{Lemma} 

\usepackage{xcolor}

\usepackage{enumitem}
\usepackage{url}
\newcommand{\figlab}[1]{\label{fig:#1}}

\newcommand{\lemlab}[1]{\label{lem:#1}}
\newcommand{\thmlab}[1]{\label{thm:#1}}
\newcommand{\figref}[1]{\ref{fig:#1}}

\newcommand{\lemref}[1]{\ref{lem:#1}}
\newcommand{\thmref}[1]{\ref{thm:#1}}
\newcommand{\hide}[1]{}
\title{Every Tetrahedron has a\\ 
$3$-vertex Quasigeodesic}

\author{Joseph O'Rourke} 

\date{\today}

\begin{document}
\maketitle

\begin{abstract}
We prove that every tetrahedron $T$ has a simple, closed
quasigeodesic that passes through three vertices of $T$.
Equivalently, every $T$ has a face whose ``exterior angles"
are at most $\pi$.
\end{abstract}

\section{Introduction}
A \emph{quasigeodesic} $Q$ is a curve on the surface of a polyhedron
that is convex to both sides in the sense that each point $p \in Q$ has surface
angle $\le \pi$ to each side.
In a vertex-free region, a quasigeodesic is a geodesic
with exactly $\pi$ to each side, but unlike geodesics,
a quasigeodesic can pass through a vertex.
Of particular interest are the simple (non-self-intesecting) closed quasigeodesics,
which we henceforth abbreviate to ``quasigeodesic" without qualifiers.
The main result of this note is that every tetrahedron $T$ has a quasigeodesic $Q$
passing through three vertices. So $Q$ is the boundary $\partial F$ of a face $F$ of $T$.
This is of interest because Pogorelov~\cite{p-qglcs-49} proved that
every convex polyhedron has at least three quasigeodesics,
a generalization of the $3$-geodesics theorem of Lysternick and Schnirlemanm.\footnote{
\url{wikipedia.org/wiki/Theorem_of_the_three_geodesics}.
See also~\cite[p.~374]{do-gfalop-07}.}
So our main theorem (Theorem~\thmref{Q3}) identifies at least one of the three guaranteed quasigeodesics on tetrahedra.

Theorem~\thmref{Q3} may also be stated without reference to the notion of a
quasigeodesic:
Every tetrahedron $T$ has a face $F$ such that the angles 
to the other side of the vertices of $F$ is at most $\pi$.
By the ``other side" is meant: the two incident face angles not in $F$,
in some sense the ``exterior angles" of $F$.
This is a fundamental 
relation among the $12$ face angles of any tetrahedron.
It could have been known since antiquity, but perhaps was of no interest
without the notion of a quasigeodesic.

\section{Notation}
The detailed argument concerning the $12$ angles requires precise notation.
\begin{figure}[htbp]
\centering
\includegraphics[width=1.0\linewidth]{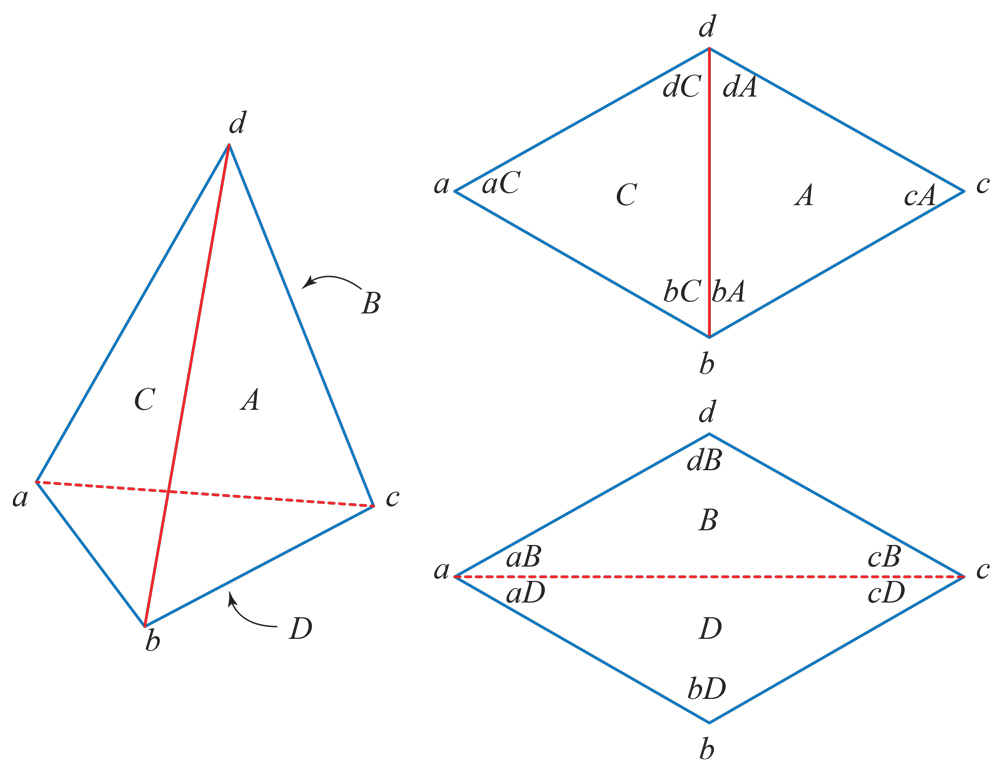}
\caption{$A=bdc$, $B=cda$, $C=adb$, $D=abc$.}
\figlab{NotationAngles}
\end{figure}

\begin{itemize}
\item Vertices of tetrahedron $T$: $a,b,c,d$.
\item Face $A$ is opposite $a$, and so does not include $a$; etc. 
\item So: $A=bdc$, $B=cda$, $C=adb$, $D=abc$. 
\item Face angles are specified by vertex and face.
So the three face angles incident to vertex $a$ are: $aB, aC, aD$; etc.
So $A$, which is opposite $a$, is not incident to $a$; etc.
See Fig.~\figref{NotationAngles}.
\item Vertex curvature: $\omega(a) = 2 \pi - (aB +  aC + aD)$.
\end{itemize}

\noindent
To simplify the calculations, angles will be represented 
in inequalities in units of $\pi$:
$1 \equiv \pi$, $2 \equiv 2\pi$, etc.
Thus under this convention, each of the $12$ face angles of a tetrahedron lies in $(0,1)$.

\section{Lemmas}
We establish two preliminary lemmas that will be used in the proof.

\begin{lem}
\lemlab{trineq}
Let $\a_1,\a_2,\a_3$ be the face angles incident to vertex $v$
of a tetrahedron $T$.
Then the angles satisfy the triangle inequality: 
$\a_1 < \a_2 + \a_3$, and similarly
$\a_2 < \a_1 + \a_3$, and $\a_3 < \a_1 + \a_2$.
The inequalities are strict unless $T$ is flat.
\end{lem}
\begin{proof}
Surround $v$ with a sphere $S$ centered on $v$. Then the planes containing the
faces incident to $v$ cut $S$ in great-circle geodesics, forming a spherical triangle on $S$.
The length of each geodesic arc is the measure of the corresponding face angle 
$\a_i$ incident to $v$. 
The triangle inequality holds on $S$, so $\a_1 < \a_2 + \a_3$,
strictly less than because the triangle cannot degenerate to a geodesic segment
(unless the tetrahedron is flat).

In the flat case, $\a_1 = \a_2 + \a_3$.
\end{proof}
\noindent
This lemma holds at any degree-$3$ vertex of a convex polyhedron.

We say that ``\emph{face $F$ fails at vertex $v$}" if the two angles incident to $v$ not in $F$
exceed $\pi$.
So, for face $A$ to fail on vertex $b$,
then among the three
face angles $bA,bC,bD$ incident to $b$, the two angles not in $A$
satisfy $bC + bD > 1$. 
This means that $\partial A$ is not a quasigeodesic, 
because to one side---the other side from $bA$---the angle exceeds $\pi$.

\paragraph{Example.}
Fig.~\figref{OneQ3} shows a tetrahedron with 
$\partial B$ a quasigeodesic, but none of
the other face boundaries is a quasigeodesic.
Its vertex coordinates are:
\begin{equation*}
a,b,c,d =
(0,0,0), \; (1,0,0), \;  (4.91,3.24,0), \; (-3.54,1.98,4.58) \;.
\end{equation*}
For example, face $B$ does not fail at vertex $a$:
$aC + aD= 125^\circ + 33^\circ = 159^\circ < \pi$.
Face $A$ fails at vertex $b$: $bC + bD = 48^\circ + 140^\circ = 188^\circ > \pi$.
\begin{figure}[htbp]
\centering
\includegraphics[width=1.0\linewidth]{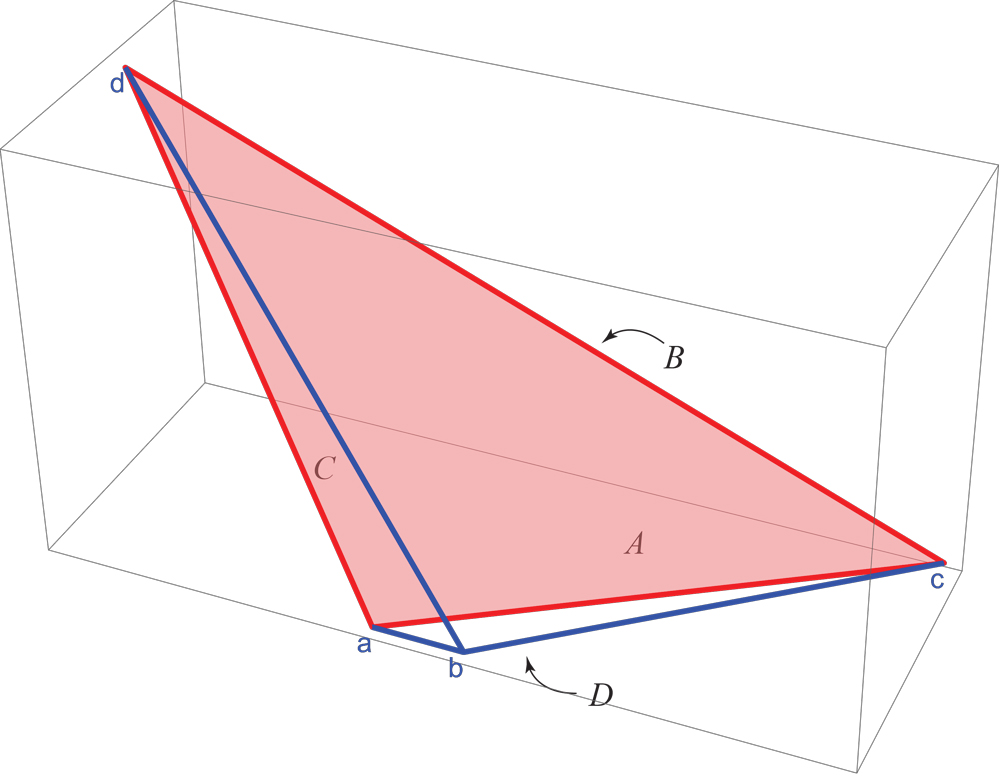}
\caption{The (red) boundary of shaded face $B=cda$ is a quasigeodesic,
but none of $\partial A, \partial C, \partial D$ are quasigeodesics.}
\figlab{OneQ3}
\end{figure}

\begin{lem}
\lemlab{curvless1}
If a face $A$ fails at a vertex $b$, then $\omega(b)<1$.
\end{lem}
\begin{proof}
Since face $A$ fails at $b$, by definition, $bC + bD > 1$.
Therefore 
\begin{eqnarray*}
\omega(b) &=& 2-(bA+bC+bD)\\
\omega(b) &=& 2-(bC+bD) - bA\\
\omega(b) &<& 1 - bA\\
\omega(b) &<& 1
\end{eqnarray*}
This establishes the claim of the lemma.
\end{proof}

\section{Case Analysis}
We now undertake a case analysis to show that it is not possible for all
four faces of tetrahedron $T$ to fail at vertices.
The cases, illustrated in Fig.~\figref{Cases}, distinguish first the number
of distinct vertices among the four face-failures, and second, the pattern
of the failures.
\begin{figure}[htbp]
\centering
\includegraphics[width=0.8\textheight]{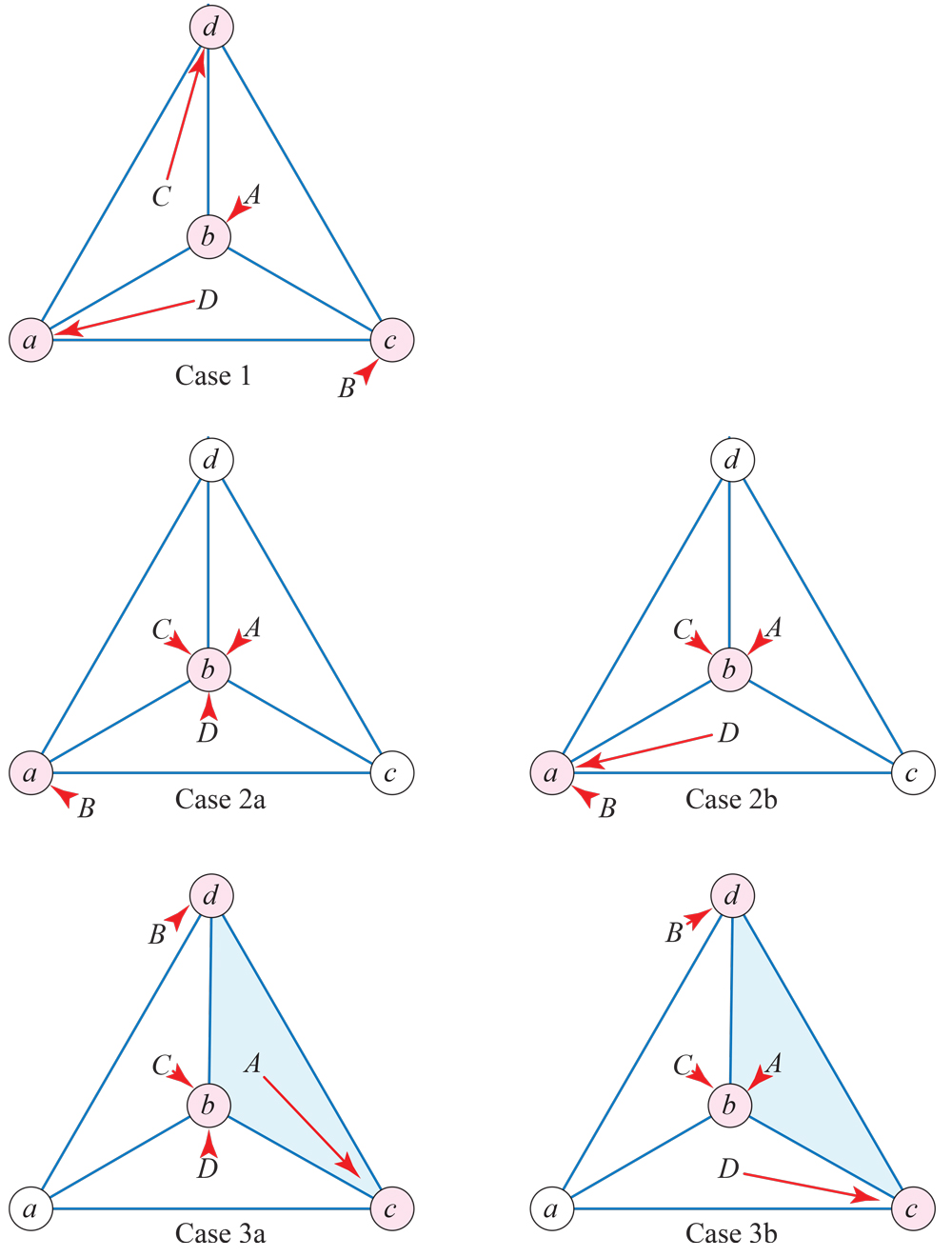}
\caption{Failures. Case~1: $4$ vertices.
Case~2: $2$ vertices. 
Case~3: $3$ vertices.}
\figlab{Cases}
\end{figure}

The proof analyzes the $12$ face angles of $T$, and shows
the set of solutions in $(0,1)^{12}$ is empty
(under the convention that each angle is in $(0,1)$).
So we are representing tetrahedra by their $12$ face angles.
The four faces each have a total of $\pi$ angle, which
reduces the dimension of the tetrahedron configuration space from $12$ to $8$.
It is known that in fact the configuration space is $5$-dimensional, not $8$-dimensional,
but the proof to follow works without including the various additional trigonometric relations
that tetrahedron angles must satisfy. It suffices to use linear equalities and inequalities
among the $12$ face angles.

\paragraph{Case 1: $4$ vertices.}
Suppose first that each of the four faces $A,B,C,D$ fail on four
distinct vertices. Then
Lemma~\lemref{curvless1} shows that
$\omega(v) < 1$ for each vertex $v$.
But then $\sum \omega(v) < 4$, contradicting the Gauss-Bonnet theorem.

\paragraph{Case 2a: $2$ vertices, $3+1$.}
Suppose now that the four faces fail on a total of two vertices.
This can occur in two distinct ways: three faces fail on one vertex,
which we call Case~2a, or two faces fail each on two vertices, Case~2b.
Say that $b$ is the vertex at which three faces fail. We then have:

\begin{eqnarray*}
A \;\mathrm{fails} \;\mathrm{at}\; b &:& bC + bD > 1 \\
B \;\mathrm{fails} \;\mathrm{at}\; a &:& aC + aD > 1 \\
C \;\mathrm{fails} \;\mathrm{at}\; b \\
D \;\mathrm{fails} \;\mathrm{at}\; b
\end{eqnarray*}
It turns out that we do not need to use the fact that
$C$ and $D$ fail at some vertices, so the implied inequalities are suppressed.
Summing the failure inequalities above leads to a contradiction:
\begin{eqnarray*}
(aC + bC) + (aD + bD) &>& 2 \\
(1 - dC) + (1 - cD) &>& 2 \\
2 &>& 2 + (dC + cD) \\
0 &>& dC + cD
\end{eqnarray*}
This is a contradiction because all angles have positive measure.

\paragraph{Case 2b: $2$ vertices, $2+2$.}
This follows the exact same proof, as again $C$ and $D$ failures are not needed
to reach a contradiction.

\paragraph{Case 3a: $3$ vertices, double outside.}
The three vertices at which faces fail bound a face, say $A$.
One vertex of $A$, say $b$, is ``doubled" in the sense that two faces fail at $b$.
Case~3a is distinguished in that neither face failing on $b$ is the three-vertex face $A$.
(Swapping $B$ to fail on $c$ and $A$ to fail of $d$
is symmetrically equivalent to the case illustrated.)

We again do not need all failures, in particular, we only need
those for faces $B$ and $D$:
\begin{eqnarray*}
A \;\mathrm{fails} \;\mathrm{at}\; c \\ 
B \;\mathrm{fails} \;\mathrm{at}\; d &:& dA + dC  > 1 \\
C \;\mathrm{fails} \;\mathrm{at}\; b  \\
D \;\mathrm{fails} \;\mathrm{at}\; b &:& bA + bC  > 1 \\
\end{eqnarray*}
\noindent
Adding these inequalities leads to the same contradiction:
\begin{eqnarray*}
(bA + dA) + (bC + dC) &>& 2 \\
(1 - cA) + (1 - aC) &>& 2 \\
0 &>& aC + cA
\end{eqnarray*}
Again a contradiction.

\paragraph{Case 3b: $3$ vertices, double inside.}
In contrast to Case~3a, in this case, one of the faces that fail on $b$ is
the three-vertex face $A$.
(Swapping $B$ to fail on $c$, $D$ to fail on $b$, and $C$ to fail on $d$,
is symmetrically equivalent.) 
This is the only difficult case, and the only case in which the
triangle inequalities guaranteed by Lemma~\lemref{trineq} are needed.

The angles of face $A$ satisfy $bA + cA + dA=1$.
Assume without loss of generality that
$bA \le cA \le dA$.
Three faces, $B,C,D$ fail at the three vertices of face $A$:
$d,b,c$ respectively.

To build intuition, we first run through the proof for specific $A$-face angles:

\begin{eqnarray*}
( bA, cA, dA ) &=& (0.1, 0.3, 0.6)\\
A \;\mathrm{fails} \;\mathrm{at}\; b \\ 
B \;\mathrm{fails} \;\mathrm{at}\; d &:& dA + dC > 1 \;:\; dC  > 0.4 \\
C \;\mathrm{fails} \;\mathrm{at}\; b &:& bA + bD > 1 \;:\; bD > 0.9 \\
D \;\mathrm{fails} \;\mathrm{at}\; c &:& cA + cB > 1 \;:\; cB  > 0.7
\end{eqnarray*}
Note $0.4 + 0.9 + 0.7 = 2$; this holds for arbitrary $A$ angles.
Now apply the triangle inequality to each of $dB, cB, dC$:
\begin{eqnarray*}
bD &<& bA + bC \;:\; bC > bD - bA \;:\; bC > 0.8 \\
cB &<& cA + cD \;:\; cD > cB -  cA \;:\;  cD > 0.4 \\
dC &<& dA + dB \;:\; dB > dC - dA \;:\;  dB > -0.2
\end{eqnarray*}
Note $0.8 + 0.4 -0.2 = 1$; this again holds for arbitrary $A$ angles.

Triangle face $D$ satisfies: $bD + cD + aD =1$.
\begin{eqnarray*}
bD &>& 0.9 \\
cD &>& 0.4 \\
bD + cD &>& 1.3 \\
bD + cD + aD &>& 1.3 \;>\; 1
\end{eqnarray*}
which contradicts $bD + cD + aD =1$.

\medskip
\begin{center}
\noindent\rule{0.5\textwidth}{0.5pt}
\end{center}
\medskip

Without specific angles assigned to $( bA, cA, dA )$, the argument
is less transparent.
Again assume that $bA \le cA \le dA$.

\begin{eqnarray*}
A \;\mathrm{fails} \;\mathrm{at}\; b \\ 
B \;\mathrm{fails} \;\mathrm{at}\; d &:& dA + dC > 1 \;:\; dC  > 1 - dA  \\
C \;\mathrm{fails} \;\mathrm{at}\; b &:& bA + bD > 1 \;:\; bD > 1 - bA \\
D \;\mathrm{fails} \;\mathrm{at}\; c &:& cA + cB > 1 \;:\; cB  > 1 - cA 
\end{eqnarray*}
Note the sum of the above three right-hand sides is $3 - (dA+bA+cA) = 2$.
\noindent
Now apply the triangle inequality to $dB, cB, dC$:
\begin{eqnarray*}
bD &<& bA + bC \;:\; bC > bD - bA \;:\; bC > 1 - 2 \cdot bA \\
cB &<& cA + cD \;:\; cD > cB -  cA \;:\;  cD > 1 - 2 \cdot cA \\
dC &<& dA + dB \;:\; dB > dC - dA \;:\;  dB > 1 - 2 \cdot dA
\end{eqnarray*}
Note the sum of the above three right-hand sides is $3 - 2(dA+bA+cA) = 1$.
\noindent
Face $D$'s angles satisfy  $bD + cD + aD =1$.
Now we reach a contradiction using the inequalities above.
\begin{eqnarray*}
bD &>& 1 - bA \\
cD &>& 1 - 2 \cdot cA \\
bD + cD &>& 2 - (bA + 2 \cdot cA )
\end{eqnarray*}
We have $(bA + 2 \cdot cA ) \le 1$ because $bA + cA + dA=1$
and $cA \le dA$.
And of course every angle is positive, so $aD > 0$.
So we have
\begin{eqnarray*}
bD + cD &>& 1\\
bD + cD + aD &>& 1
\end{eqnarray*}
which contradicts $bD + cD + aD =1$.

That the inequalities for each of the above cases
cannot be simultaneously satisfied
has been verified by Mathematica's
\texttt{FindInstance[]} function, which
uses Linear Programming over the rationals\footnote{
\url{https://mathematica.stackexchange.com/q/255494/194}}
to conclude that the set of 
solutions in $\mathbb{R}^{12}$ is empty.

Replacing the triangle inequalities with equalities when the tetrahedron
is flat (e.g., $aB = aC + aD$ instead of $aB < aC + aD$) again leads to the same contradiction.

\section{Conclusion}
\begin{thm}
\thmlab{Q3}
Every tetrahedron has at least one face $F$ whose boundary
$\partial F$ is a simple, closed quasigeodesic $Q$,
passing through the three vertices of $F$.
So $Q$ is a $3$-vertex quasigeodesic.
\end{thm}


In Open Problem~18.13~\cite{Reshaping},
we conjecture that every convex polyhedron either has a simple closed geodesic,
or a simple closed quasigeodesic through just one vertex,
i.e., a $1$-vertex quasigeodesic.
This remains for future work.

\bibliographystyle{alpha}
\bibliography{refs}

\end{document}